\documentclass[11pt,a4paper]{article}
\usepackage{amsthm}
\usepackage{amsfonts, amssymb, amsmath, amscd, latexsym}
\usepackage{eucal, enumerate, latexsym}
\usepackage{graphics}
\usepackage{graphicx}

\makeatletter \@addtoreset{equation}{section}

\makeatother \makeatletter

\def\dis{\displaystyle}

\def\sm{\smallskip}

\def\L2{L^2(\Omega;\mathbb{R})}

\def\beq{\begin{eqnarray}}
\def\eeq{\end{eqnarray}}
\def\qed{\hfill$\Box$}

\newtheorem{theorem}{Theorem}[section]

\newtheorem{proposition}{Proposition}[section]
\newtheorem{defn}{Definition}[section]
\newtheorem{remark}{Remark}[section]
\newtheorem{lemma}{Lemma}[section]
\newtheorem{Ex}{Example}[section]

\begin{document}

\title{Nonlocal solutions of parabolic equations \\with strongly elliptic differential operators}
\author{
Irene Benedetti\\
\small\textit{Department of Mathematics and Computer Sciences,
University of Perugia}\\
\small\textit{I-06123 Italy, e-mail: irene.benedetti@dmi.unipg.it}\\
\and Luisa Malaguti\\
\small\textit{Department of Sciences and Methods for Engineering, University of Modena and Reggio Emilia}\\
\small\textit{I-42122 Italy, e-mail: luisa.malaguti@unimore.it}
\and Valentina Taddei\\
\small\textit{Department of Sciences and Methods for Engineering, University of Modena and Reggio Emilia}\\
\small\textit{I-42122 Italy, e-mail: valentina.taddei@unimore.it}}
\date{}
\maketitle
\begin{abstract} The paper deals with second order parabolic equations on bounded domains with Dirichlet conditions in arbitrary Euclidean spaces. Their interest comes from being models for describing reaction-diffusion processes in several frameworks. A linear diffusion term in divergence form is included which generates a strongly elliptic differential operator. A further  linear  part, of integral type, is present which accounts of nonlocal diffusion behaviours. The main result provides a unifying method for studying the existence and localization of solutions satisfying nonlocal associated boundary conditions. The Cauchy multipoint and the mean value conditions are included in this investigation. The problem is transformed into its abstract setting and the proofs are based on the homotopic invariance of the Leray-Schauder topological degree. A \emph{bounding function} (i.e. Lyapunov-like function) theory is developed, which is new in this infinite dimensional context. It allows that the associated vector fields have no fixed points on the boundary of their domains and then it makes possible the use of a degree argument.

\vspace{0.5 cm}
\noindent \textbf{AMS Subject Classification:} Primary 35K20. Secondary 34B10, 47H11, 93D30.

\smallskip
\noindent
\textbf{Keywords:} Parabolic equations; multipoint and mean value conditions; degree theory; Lyapunov-like functions.
\end{abstract}
\section{Introduction}\label{s:intro}
The paper deals with the second order parabolic equation
\begin{equation}\label{e:PE}
\frac{\partial u(t, \xi)}{\partial t}= \displaystyle{\sum_{i,j=1}^n\frac{\partial}{\partial \xi_i} \left( a_{i,j}(\xi)\frac{\partial u(t,\xi)}{\partial \xi_j}\right)+\int_{D}k(\xi,y)u(t,y)dy}-bu(t,\xi)+g(t,u(t,\xi))
\end{equation}
with $t\in [0,T]$  and $\xi\in D\subset \mathbb{R}^n$, where  $D$  is a bounded domain with a sufficiently regular boundary $\partial D$.  The coefficients $a_{i,j}\in C^1(\overline D) $ for $i,j=1,...,n$, are symmetric i.e.
\begin{equation}\label{e:sym}
a_{i,j}(\xi)=a_{j,i}(\xi), \, \xi\in \overline D \quad \text{for }i, j =1,...,n \text{ with } i\ne j
\end{equation}
and there is a value $C_0>0$ such that
\begin{equation}\label{e:elliptic}
 C_0\Vert \sigma \Vert^2 \le \displaystyle{\sum_{i,j=1}^n}a_{i,j}(\xi)\sigma_i\sigma_j  \text{ for all } \sigma \in \mathbb{R}^n.
\end{equation}
Moreover $b > 0$ is a prescribed constant, $k: D \times D \to \mathbb{R}$ and $g:[0,T]\times \mathbb{R} \to \mathbb{R}$ are two given maps. The solution is subject to the Dirichlet boundary conditions
\begin{equation}\label{e:D}
u(t,\xi)=0, \quad \text{for } t\in[0,T], \, \xi \in \partial D.
\end{equation}

Equation \eqref{e:PE} is a model for reaction-diffusion processes in many frameworks and hence it is widely investigated. We refer to the recent monographs \cite{DB G V}, \cite{GK}, \cite{MN} and \cite{Yagi} for a wide discussion on parabolic dynamics. The symmetric second order differential operator in its r.h.s.  accounts of diffusion behaviours of a punctual type while the nonlocal term in integral form includes long distance diffusive interactions or memory effects. When $a_{i,j}(\xi) \equiv \delta_{i,j}=\left\{\begin{array}{ll} 0 & i\ne j \\
1 & i=j \end{array}\right.$,  the differential term on the right hand side of \eqref{e:PE} simply reduces to the Laplace operator and hence \eqref{e:PE} becomes
\begin{equation*}
u_t(t,\xi)=\Delta u(t,\xi)+\int_Dk(\xi, y)u(t, y)\, dy -bu(t,\xi)+g(t, u(t,\xi)), \enspace t\in[0,T], \, \xi \in D.
\end{equation*}

We always assume that
\begin{equation}\label{e:gandk}
\begin{array}{rl}
(i)&g  \text{ is continuous and there exist } L>0 \text{ and } \beta \in (0,1) \text{ such that}\\
&\vert g(t, \xi)-g(t, y)\vert \le L\max\{\vert \xi-y\vert^{\beta}, \vert \xi-y\vert \}, \text{ for } t \in [0,T], \, \xi,y \in \mathbb{R},\\[5mm]
(ii)&k\in L^{\infty}(D \times D) \text{ and } 0\le k(\xi,y) \le 1 \text{ for a.a. }  \xi,y \in D.
\end{array}
\end{equation}
By the estimate in \eqref{e:gandk}\emph{(i)} the function $g$  has a sublinear growth in its second variable $\xi$ when $\vert \xi \vert \to \infty$, for every $t\in [0,T]$;  $g$  is also  H\"{o}lder continuous with exponent $\beta$ in $\xi$ for every $t$ and then, in particular,   $g(t,\xi)$ may approach $g(t,0)$ as $\vert \xi \vert^{\beta}$ when $\xi \to 0$.

\noindent When $k(\xi, y):=h(\xi-y)$ for a.a. $\xi, y \in D$ with $h \in L^{\infty}(D)$ and $0\le h(\xi)\le 1$ for a.a. $\xi \in D$, the integral term in \eqref{e:PE} can be written in the form $h\ast u(t,\cdot)$, i.e. is a convolution product with convolution kernel $h$.

\smallskip As usual $L^p(D)$ denotes the Lebesgue space  $L^p(D, \mathbb{R})$ and we always  restrict to the case when
\begin{equation*}
1<p<\infty.
\end{equation*}
Under conditions \eqref{e:sym} and \eqref{e:elliptic} the linear elliptic partial differential operator in divergence form $A_p \colon W^{2,p}\left(D\right)\cap W^{1,p}_0\left(D\right)\to
L^p\left(D\right)$ given by
\begin{equation}\label{e:divergence}
A_p(v)(x)=\sum_{i,j=1}^n \frac{\partial}{\partial x_i}\left( a_{ij}(x)\frac{\partial v(x)}{\partial x_j}\right)
\end{equation}
is well-defined and it is the infinitesimal generator of an analytic semigroup of contractions $\{S(t)\}_{t\ge 0}$ in $L^p(D)$ (see e.g. \cite[Theorem 3.6 p. 215]{p}); we refer to  Section \ref{s:prelim} for an additional discussion about this semigroup.

\noindent The abstract formulation of \eqref{e:PE} takes  the form
\begin{equation}
\label{e:AbEq}
\begin{array}{l}
x^{\, \prime}(t)=Ax(t)+f(t,x(t)), \quad t\in [0,T], \, x\in L^p(D)\\
\end{array}
\end{equation}
with
\begin{equation}\label{e:fromutox}
x(t):=u(t,\cdot), \quad \text{for } t\in [0,T].
\end{equation}
 As usual \eqref{e:AbEq} is obtained by equation \eqref{e:PE} when  $u$ is no longer considered as a function of the variables $t$ and $\xi$, but as a mapping $t \longmapsto u(t, \cdot)$ with $t \in [0,T]$ and $u(t, \cdot)$ in a suitable function space. For short in \eqref{e:AbEq} we simply denote with $A$ the linear operator  $A_p$ while the function $f \colon [0,T]\times L^p(D) \to L^p(D)$ includes all the additional terms in \eqref{e:PE} (see formula \eqref{e:fH} in Section \ref{s:application}).

 \noindent By a solution of \eqref{e:PE} we mean a function $u\colon [0,T]\times D \to \mathbb{R}$ with
$u(t, \cdot) \in L^p(D)$ for all $t \in [0,T]$ such that the corresponding function $x$ introduced in \eqref{e:fromutox} belongs to   $C([0,T], L^p(D))$ and it is a solution of \eqref{e:AbEq} in integral form, i.e. $x$ is a mild solution (see Definition \eqref{d:mild}) of \eqref{e:AbEq}.

\smallskip

The  existence of solutions to \eqref{e:PE} which satisfy given nonlocal conditions displays a growing interest for the possibility of these trajectories to capture additional information about the dynamics. We mention, for instance, the mean value condition.
\begin{equation}\label{e:mean}
u(0,\xi)=\frac{1}{T}\int_0^T u(t,\xi)\, dt, \quad \text{for a.a. }\xi \in D
\end{equation}
and the multipoint condition
\begin{equation}\label{e:Cauchy}
u(0,\xi)=\displaystyle{\sum_{i=1}^{q}}\alpha_i u(t_i, \xi) \text{ with }t_i \in (0,T], \, \alpha_i \in \mathbb{R}, \, i=1,...,q \quad \text{and a.a. }\xi\in D.
\end{equation}

\smallskip
The linear parabolic case with no integral term and boundary conditions as in \eqref{e:Cauchy} is in Chabrowski \cite{Cw}; the study is based on a maximum principle and the use of a Green function. The model in Deng \cite{Deng} deals with the evolution of a small quantity of gas in a tube; the nonlocal condition is of integral type (see \eqref{e:mean}) and $t$  varies on a half-line; the nonlinear term is smooth and also the asymptotic behaviour of the solution at infinity is discussed. The nonlocal condition in Jackson \cite{J} is quite general and possibly nonlinear. Pao \cite{pao} treated the existence and multiplicity of solutions between a pair of ordered upper and lower solution again in a smooth model which also  includes a nonlocal  initial condition. Infante-Maciejewski \cite{IM} and Xue \cite{xue09} studied  systems of two equations with an elliptic part  given by the Laplace operator; while the latter is based on a fixed point argument, in the former a degree argument is used and the appearance of positive solutions is proved; strong growth restrictions on the  terms are assumed in both papers. The model introduced by Zhu-Li \cite{zhuli11} is quite general, again with integral nonlocal conditions, but the growth and regularity conditions are rather strong and given in implicit form. A degree argument is used also by Benedetti-Loi-Taddei  \cite{BLT}, combined with an approximation solvability method and it seems especially useful for treating the case when the nonlinearity depends on some weighted mean value of the solution.  At last the model proposed by Viorel \cite{Viorel} has an autonomous nonlinearity of polynomial type with a superlinear growth at infinity.

\bigskip
We will prove the following result on  the existence of solutions satisfying the above boundary conditions. Notice that our model has quite general regularity conditions and no growth restrictions on its term. As usual the symbol  $\vert D\vert$ denotes the Lebesgue measure of the set $D$.
\begin{theorem}\label{t:nonPEP}
Consider equation \eqref{e:PE} with $a_{i,j}\in C^1(\overline D), \, i,j=1,...,n$ satisfying  \eqref{e:sym}, \eqref{e:elliptic}. Assume conditions in \eqref{e:gandk} and let $b>L+\vert D\vert$. Then problem \eqref{e:PE}-\eqref{e:D}

\begin{itemize}
\item[\emph{(i)}] admits a solution satisfying condition \eqref{e:mean};
\item[\emph{(ii)}] admits a solution satisfying condition \eqref{e:Cauchy} provided that
\begin{equation*}
\sum_{i=1}^q \vert \alpha_i\vert \le 1.
\end{equation*}
\end{itemize}

\end{theorem}
The paper contains a wider discussion which involves the quite general nonlocal condition
\begin{equation}\label{e:nonPE}
x(0)=M(x)
\end{equation}
where $M\colon C([0,T], L^p(D)) \to L^p(D) $ and $x$ is the function defined in \eqref{e:fromutox}. It is clear that \eqref{e:mean}, \eqref{e:Cauchy} and the Cauchy condition $x(0)=x_0=u(0, \cdot)$ satisfy \eqref{e:nonPE}.  Notice moreover that
\begin{itemize}
\item[(i)] the periodic condition: $M(x)=x(T)=u(T, \cdot); $
\item[(ii)] the antiperiodic condition: $ M(x)=-x(T)=u(T, \cdot) $
\end{itemize}
are special cases of \eqref{e:Cauchy}. We remark that \eqref{e:nonPE} also includes nonlinear conditions such as
\begin{equation*}
u(0,\xi)=G\left( \int_0^T h(t)u(t,\xi)\, dt\right), \quad \xi \in D
\end{equation*}
with suitable $h \colon [0,T] \to \mathbb{R}$ and $G \colon \mathbb{R} \to \mathbb{R}$ introduced for instance in \cite{BTV} (see Example 5) in the framework of age-population models.

 \smallskip

 The proof of Theorem \ref{t:nonPEP} is in Section \ref{s:application} where there is a quite general discussion about problem\eqref{e:PE}-\eqref{e:D}-\eqref{e:nonPE} (see Theorem \ref{t:PEgen}). The results are based on a unifying approach of topological type on the abstract setting, i.e. for equation \eqref{e:equation} (see Section \ref{s:abstrat});  a degree argument, in particular, is used there which makes then possible to avoid strong restrictions on the terms of \eqref{e:PE} as already noted about Theorem \ref{t:nonPEP}. On the other hand, the involved  vector fields need to be  fixed-points free on their boundary; the  property is obtained by a bounding function (i.e. Lyapunov-like) method which is original in this infinite dimensional setting; the method is discussed in Section \ref{s:bounding}. Section \ref{s:prelim} contains some notation and preliminary results.
\smallskip

\smallskip

Several Banach spaces appear in this paper; we simply use the symbol $\Vert \cdot \Vert$ to denote the norm in all of them when it is clear from the context  which is, each time, the involved space. The symbol $ E^* $ stands for the dual space of $ E. $

\section{Preliminary results and notation}\label{s:prelim}
\noindent Let  $E$ be a Banach space. A family of linear, bounded operators $S(t):E\to E$, for $t $ in the interval $[0, \infty)$, is called a \emph{$C_0$-semigroup} if the following conditions are satisfied:
\begin{enumerate}[(a)]
\item $S(0)=I$;
\item $S(t+r)=S(t)S(r)=S(r)S(t)$ for  $t,r \in [0, \infty)$;
\item the function $t\to S(t)x$ is continuous on $[0, \infty)$, for every $x\in E$.
\end{enumerate}
The \emph{infinitesimal generator } of $S(t)$ is the linear operator $A$ defined by
\begin{equation*}
Ax=\lim_{h\to 0^+}\frac {\left ( S(h)-I\right)x}{h}, \, x\in D(A)
\end{equation*}
with
\begin{equation*}
D(A):=\left\{x\in E \, : \, \lim_{h\to 0^+}\frac {\left ( S(h)-I\right)x }{h} \text{ exists }\right\}.
\end{equation*}
We refer to \cite{Lun}, \cite{p}, \cite{vr2} for the theory of semigroups. Here we only restrict to  those properties which are needed in our investigation.

\noindent As a straightforward consequence of \emph{(c)} (see e.g. \cite[Theorem 2.4 p.4]{p}), we obtain
\begin{equation}\label{e:media}
\lim_{h\to 0}\frac{1}{h}\int_t^{t+h}S(s)x \, ds=S(t)x, \enspace t\ge 0, \, x \in  E.
\end{equation}

\smallskip
\noindent Every $C_0$-semigroup is bounded, for $t$ in a bounded interval, (see e.g. \cite[Theorem 2.2 p.4]{p}), in the space $\mathcal{L}(E)$ of linear, bounded operators.  When further $\Vert S(t) \Vert \leq 1 $ for  $ t \geq 0, \{S(t)\}_{t\ge 0}$ is said to be a \emph{contraction semigroup}. It is easy to see that every contraction semigroup satisfies
\begin{equation}\label{e:contraction}
S(t)r\overline B \subseteq r\overline B, \qquad \text{for  } r>0, \,  t\ge0,
\end{equation}
where $B$ is the open unit ball in $E$ centered at $0$. The  semigroup $\{S(t)\}_{t\ge 0}$ is said:
\begin{enumerate}[]
\item \emph{compact} if $S(t)$ is compact, for $t>0$;
\item \emph{uniformly differentiable} if the map $t \longmapsto S(t)$ defined on $[0, \infty)$ with values in $\mathcal{L}(E)$ is differentiable for every $t>0$ (see e.g. \cite[Definition 6.3.2]{vr2}).
\end{enumerate}
For $0<\theta\le \pi$ define the sector
\begin{equation*}
C_{\theta}:=\{z\in \mathbb{C} \, : \, -\theta <\text{arg } z <\theta\}.
\end{equation*}
Clearly
\begin{equation*}
\overline{C}_{\theta}=\{z\in \mathbb{C} \, : \, -\theta \le \text{arg } z \le \theta\}\cup \{0\}.
\end{equation*}
\begin{defn}\label{d:analytic} (\cite[Definition 7.1.1]{vr2}) The  $C_0$-semigroup $\{S(t)\}_{t\ge 0}$ is said to be \emph{analytic}  if there is $\theta \in (0, \pi]$  and a mapping $\tilde S \colon \overline{C}_{\theta}\to \mathcal{L}(E) $ such that
\begin{enumerate}[(i)]
\item $S(t)=\tilde S(t), \, t\ge 0; $
\item $\tilde S(z+w)=\tilde S(z)\tilde S(w)$ for  $z, w \in \overline{C}_{\theta}; $
\item $\displaystyle{\lim_{z\in \overline{C}_{\theta}, \,  z\to 0}}\tilde S(z)x=x$ for $x \in E; $
\item the mapping $z \longmapsto \tilde S(z)$ is analytic from $\overline{C}_{\theta}$ to $\mathcal{L}(E). $
\end{enumerate}
\end{defn}
\begin{lemma}\label{l:graphnorm}(\cite[Corollary 3.5.1]{vr2} \ Let $A\colon D(A) \subseteq E \to E$ be the infinitesimal generator of a C$_0$-semigroup. The map $\Vert \cdot \Vert_{D(A)} \colon D(A) \to \mathbb{R}$ given by
\begin{equation}
\Vert x\Vert_{D(A)} :=\Vert x \Vert+\Vert Ax \Vert
\end{equation}
defines a norm in $D(A)$, called graph norm, with respect to which $D(A)$ is a Banach space.
\end{lemma}

\noindent We make use, in the sequel, of the following compactness condition which involves the graph norm.

\begin{proposition}\cite[Corollary 6.3.2]{vr2}\label{p:Vrabie} \ Let  $A \colon D(A) \subseteq E \to E$ be the infinitesimal generator of a uniformly differentiable C$_0$-semigroup of contractions. If $D(A)$, endowed with the graph norm, is compactly embedded in $E$, then the semigroup generated by $A$ is compact.
\end{proposition}

\smallskip
 We complete this brief discussion about semigroup theory with an important result about the semigroup generated by the  elliptic operator introduced in \eqref{e:divergence}.
\begin{theorem}\label{t:compact} \ Consider the  linear operator $A_p$ defined in \eqref{e:divergence}.  When conditions  \eqref{e:sym}, \eqref{e:elliptic} are satisfied, the semigroup generated by $A_p$ is compact.
\end{theorem}
\begin{proof}
 The linear operator $A_p$ is the infinitesimal generator of an analytic semigroup of contractions on $L^p(D)$ (see e.g. \cite[Theorem 3.6 p.215 ]{p}). In particular, the semigroup is uniformly  differentiable. $A_p$ is also strongly elliptic of order $2$ by conditions \eqref{e:sym}, \eqref{e:elliptic}. Therefore, the following estimate holds  (see e.g \cite[Theorem 3.1 p. 212]{p})
\begin{equation*}
\Vert v\Vert_{W^{2,p}(D)}\le C\left( \Vert Av\Vert_{L^p(D)}+\Vert v\Vert_{L^p(D)}\right)=Cv\Vert_{D(A_p)}, \quad v \in D(A_p),
\end{equation*}
for some $C>0$. Hence, the Banach space $\left(D(A_p), \Vert \cdot \Vert_{D(A_p)}\right)$ is continuously embedded into the space $\left(D(A_p), \Vert \cdot \Vert_{W^{2,p}(D)}\right)$. By Sobolev-Rellich-Kondrachov Theroem (see e.g. \cite[Theorem 1.5.4]{vr2}) the embedding of $\left(D(A_p), \Vert \cdot \Vert_{W^{2,p}(D)}\right)$ into $\left(D(A_p), \Vert \cdot \Vert_{L^p(D)}\right)$ is compact. We obtained  that the embedding of $\left(D(A_p), \Vert \cdot \Vert_{D(A_p)}\right)$ into $\left(D(A_p), \Vert \cdot \Vert_{L^p(D)}\right)$ is compact. We complete the proof by means of Proposition \ref{p:Vrabie}.
\end{proof}

\smallskip
\begin{defn}\label{d:semi}(\cite[Definition 4.2.1]{KOZ}) \ A sequence $\{f_n\}\subset L^1([a,b], E)$ is said to be \emph{semicompact} if it is integrably bounded, i.e. $ \Vert f_n(t) \Vert \le \nu(t) $ for a.a. $ t \in [a,b] $ and all $n \in \mathbb{N} $ with $ \nu\in L^1([a,b], E)$,  and the set $\{f_n(t)\}$ is relatively compact for a.a. $t \in [a,b]$.
\end{defn}

We recall now a useful compactness result in the space of continuous function, involving semicompact sequences.

\begin{theorem}\label{t:semicomp} (\cite[Theorem 5.1.1.]{KOZ}) \ Let $G\colon L^1([a,b], E) \to C([a,b], E) $ be an operator satisfying the following conditions
\begin{itemize}
\item[{(i)}] \ there exists $\sigma \ge 0$ such that
\begin{equation*}
\Vert Gf - Gg \Vert \le \sigma \Vert f-g \Vert
\end{equation*}
\item[{(ii)}] \ for any compact $K \subset E$ and sequence $\{f_n\}\subset L^1([a,b], E)$ such that $\{f_n(t)\} \subset K$ for a.a. $t \in [a,b]$, the weak convergence $f_n \rightharpoonup f_0$ in $L^1([a,b], E)$ implies that $Gf_n \to Gf_0$.
\end{itemize}
Then, for every semicompact sequence $\{f_n\}\subset L^1([a,b], E)$ the sequence $\{Gf_n\}$ is relatively compact in $C([a,b], E) $ and, moreover, if $f_n \rightharpoonup f_0$ then $Gf_n \to Gf_0$.
\end{theorem}

\begin{Ex} \label{ex:CO} \ Let $\{S(t)\}_{t\ge 0} $ be a (not necessarily compact) $C_0$-semigroup. Then, the associated Cauchy operator $G \colon L^1([a,b], E) \to C([a,b], E)$ defined by
$$
Gf(t)=\int_a^t S(t-s)f(s) \, ds
$$
satisfies conditions (i) and (ii) in Theorem \ref{t:semicomp} (see e.g. \cite[Lemma 4.2.1]{KOZ}).
\end{Ex}

\bigskip
A function $f \colon X\to Y$ between the Banach spaces $X$ and $Y$ is said to be completely continuous if it is continuous and maps bounded subsets $U\subset X$ into relatively compact subsets of $Y$. Given a nonempty, open and bounded set $U\subset X$ and a completely continuous map $g\colon\overline{U}\to X$ satisfying $x\neq g(x)$ for all $x\in\partial U$, then for the corresponding vector field $i-g$ (where $i$ denotes the identity map on $ X $) the Leray-Schauder topological degree $deg (i-g,\overline{U})$ is well-defined (see, e.g. \cite{KrZa, LS}) and it satisfies the usual properties.

\section{Nonlocal solutions in Banach spaces}\label{s:abstrat}
In this part we deal with the equation \eqref{e:AbEq}. Indeed, in order to lead a discussion as general as possible, we let $t$ varying in an arbitrary interval $[a,b]$ and $x$ in a reflexive Banach space $E$ i.e. we consider
\begin{equation}\label{e:equation}
x^{\, \prime}(t)=Ax(t)+f(t,x(t)), \quad t\in [a,b], \, x\in E.
\end{equation}
We assume that
\begin{itemize}
\item[\emph{(A)}] \vspace{.2cm} $ A $ is a
linear, not necessarily bounded, operator with $ A: D(A) \subset E \rightarrow E $ and it generates a compact $C_0$-semigroup of contractions $S: [0,\infty) \to \mathcal{L}(E)$ (see Section \ref{s:prelim} for details);
\item[\emph{(f)}] \vspace{.2cm} the function $f\colon [a,b]\times E \to E$ is  continuous and for every $\Omega \subset E$ bounded there is $\nu_{\Omega} \in L^1([a,b])$ such that $\Vert f(t,x) \Vert \le \nu_{\Omega}(t)$ for a.a. $t \in [a,b]$ and  $x \in \Omega$.
\end{itemize}

\smallskip

We consider mild solutions of \eqref{e:equation}, that is functions $x \in C([a,b],E)$ which satisfy \eqref{e:equation} in integral form; more precisely
\begin{defn}\label{d:mild} A function $x\in C([a,b], E)$ is said to be a mild solution of the equation \eqref{e:equation} if
\begin{equation}
x(t)=S(t-a)x(a)+\int_a^t S(t-s)f(s, x(s))\, ds, \quad  t \in [a,b].
\end{equation}
\end{defn}

We combine equation \eqref{e:equation} with a nonlocal condition
\begin{equation}
\label{e:nonlocal}
\begin{array}{l}
x(a)=M(x),
\end{array}
\end{equation}
where $M \colon C([a,b], E) \to E. $ We assume that (see below formula \eqref{e:contraction} for the symbol $B$)
\begin{itemize}

\item[(\emph{m$_0$})] \vspace{.2cm} there exists a positive constant $ r $ such that $ M\left( C([a,b], r\overline B)\right) \subseteq r\overline B; $
\item[(\emph{m$_1$})] \vspace{.2cm} if $ \{x_n\} \subset C([a,b], r\overline B) $ and $x_n(t) \to x(t)$ for $t \in (a,b]$ with $x \in C([a,b], E), $ then $M(x_n)\to M(x)$ as $n \to \infty$.
\end{itemize}

\smallskip
\begin{remark}\label{r:M}
 By condition (m$_1$), the function $ M $ is clearly continuous, when restricted to $ C([a,b], r\overline B). $

\end{remark}

\smallskip
The study of problem \eqref{e:equation}-\eqref{e:nonlocal} very naturally leads to be performed with a topological method. It was initiated by Byszewski \cite{BY} and the nonlocal condition there is of the type
\begin{equation*}
x(a)+g(t_1, ..., t_p, x(t_1), ..., x(t_p))=x_0, \quad a<t_1<...<t_p\le b,
\end{equation*}
hence possibly nonlinear. Additional results can be found in \cite{BLT}, \cite{BP}, \cite{CR}, \cite{KOWY}, \cite{LLX}, \cite{xue09} and \cite{ZSL12} (see also the references there). A fixed point theorem is used in all these papers such as the Banach contraction principle, the Schauder fixed point theorem or the fixed point theorem for condensing maps; hence, strong growth and regularity assumptions are needed,  such as the compactness of the function involved in the nonlocal condition or the sublinearity of the nonlinear term with respect to the variable $x$.

\noindent A topological degree was introduced  by \'{C}wiszewski-Kokocki \cite{CK} (see also \cite{Koko}) for the study of a periodic problem. A new approximation solvability method, involving a degree argument, was used in Benedetti-Loi-Taddei \cite{BLT}, it allows to treat nonlinear terms satisfying the very general growth condition in \emph{(f)} but requires the continuity of $f(t, \cdot)$ with respect to the weak topology in $E$ for a.a. $t$ and the linearity of $M$. Despite, like the above mentioned results, in this paper we use the invariance of an appropriate topological degree by an homotopic field, our new technique allows to assume the very general growth condition \emph{(f)} and we do not need any compactness or linearity conditions on $M$. However, in order that such an invariance is satisfied, the vector fields need to be fixed-points free on the boundary of their domains (see Section \ref{s:prelim}). This is usually known as the \emph{transversality condition}  which is strictly related to the notion of bounding function (i.e. Lyapunov-like function, see Definition \eqref{d:bf} and Section \ref{s:bounding}).

\smallskip \noindent In some cases (see e.g. \cite{AOmair}, \cite{CR}, \cite{KOWY} and \cite{xue09}) the discussion took place in the multivalued setting. We claim that, with minor changes, the present investigation can be generalized to  multivalued dynamics.

 \begin{defn}\label{d:bf} Let $K \subset E$ be nonempty, open and bounded. A function $V \colon E \to \mathbb{R}$ satisfying
\begin{itemize}
\item[\emph{(V1)}] $ V/_{\partial K}=0, \quad V/_{K}\le 0$;
\item[\emph{(V2)}] $\dis{\liminf_{h \to 0^-}}\frac{V(x+hf(t,x))}{h}<0$, for all $t \in (a,b]$ and $x \in \partial K$;
\end{itemize}
is said to be a bounding function for equation \eqref{e:equation}.
\end{defn}

\noindent By means of this tool we can use the Leray-Schauder degree theory in order to prove the following result.

\begin{theorem}\label{t:main} \ Consider the b.v.p. \eqref{e:equation}-\eqref{e:nonlocal} under conditions \emph{(A)}, \emph{(f)}, \emph{(m$_0$)} and \emph{(m$_1$)}. Let there exists a locally Lipschitzian bounding function $V \colon E \to \mathbb{R}$  of \eqref{e:equation} with $K:=rB $ and $ r $ as in \emph{(m$_0$)}.

\noindent Then problem \eqref{e:equation}-\eqref{e:nonlocal} admits at least one mild solution $x \in C([a,b], r\overline B)$.
\end{theorem}

 \begin{proof}The proof splits into two parts. By means of the homotopic invariance of the Leray-Schauder degree, we first solve an initial value problem associated to equation \eqref{e:equation} in an arbitrary interval $[a+\frac 1m, b]$ with $m \in \mathbb{N}$ sufficiently large (see \eqref{e:x}-\eqref{e:x(a+1/m)} with $\lambda=1$). The compactness of the semigroup $S(t)$ is fundamental in this reasoning and this is why we restrict to the interval  $[a+\frac 1m, b]$. In such a way we get a sequence $\{x_m\}$ of continuous functions taking values in $r\overline B$. Then we obtain a solution of the original problem \eqref{e:equation}-\eqref{e:nonlocal} by passing to the limit in the sequence $\{x_m\}$. Let $Q:=C([a,b], r\overline B)$ with $r>0$ as in \emph{(m$_0$)} and $m \in \mathbb{N}$ such that $a+\frac 1m <b. $  Notice that, since $ \{S(t)\}_{t\ge 0} $ is a semigroup of contractions, it follows that  $ S(t)r\overline B\subseteq r\overline B, \, t \ge 0.$

\smallskip

\noindent STEP 1. \ We define the map $\mathcal{T}_m \, :\, Q\times [0,1]\to C([a,b], E )$ as follows
\begin{equation}\label{e:DTm}
\mathcal{T}_m(q, \lambda)(t)=\left\{
\begin{array}{ll}\lambda S(\frac 1m)M(q) & t\in [a, a+\frac 1m]\\
\lambda S(t-a)M(q)+\lambda\int_{a+\frac 1m}^t S(t-s)f(s,q(s))ds & t\in[a+\frac 1m, b],
\end{array}
\right.
\end{equation}
which, according to \emph{(A)} and \emph{(f)}, is well defined. \\
With $q, \, \lambda $ and $ m $ as before, let $ y_{q,\lambda}, \eta_{q,\lambda} \colon [a+\frac 1m, b] \to E $ be given by
\begin{equation}\label{e:yeta}
y_{q,\lambda}(t):=\lambda S(t-a)M(q), \qquad \eta_{q,\lambda}(t):=\lambda\int_{a+\frac 1m}^t S(t-s)f(s,q(s))ds;
\end{equation}
notice that
\begin{equation*}
\mathcal{T}_m(q, \lambda)(t)=y_{q,\lambda}(t)+ \eta_{q,\lambda}(t), \qquad t\in[a+\frac 1m, b].
\end{equation*}
By the equality
\begin{equation*}
S(t-\frac 1m -a)\left[\lambda S(\frac1m) M(q)\right]= \lambda S(t-a)M(q), \qquad t\in[a+\frac 1m, b]
\end{equation*}
we obtain that the function $x:=\mathcal{T}_m(q, \lambda)$ is the unique solution (see \cite[Corollary 2.2 p.106]{p}) of the linear initial value problem
\begin{equation}\label{e:ivpz}
\left\{\begin{array}{rl}
z'(t)=&Az(t)+\lambda f(t,q(t)), \enspace t\in[a+\frac 1m,b]\\
z(a+\frac 1m)=&\lambda S(\frac 1m) M(q).
\end{array}
\right.
\end{equation}
In particular, every fixed point $x=\mathcal{T}_m(x, \lambda), $ with $ x\in Q $ and $\lambda \in [0,1] $ is a mild solution of the equation
\begin{equation}\label{e:x}
x^{\prime}(t)=Ax(t)+\lambda f(t,x(t)), \qquad t\in [a+\frac 1m, b]
\end{equation}
(see Definition \ref{d:mild}) which satisfies
\begin{equation}\label{e:x(a+1/m)}
x(a+\frac 1m)=\lambda S(\frac 1m)M(x).
\end{equation}
We will show that $\mathcal{T}_m(\cdot, 1)$ has a fixed point $x_m=\mathcal{T}_m(x_m, 1), $ with $x_m \in Q. $\\

The use of a topological method then arises quite naturally and hence we investigate, in the following, the regularity properties of the map $\mathcal{T}_m$.

\smallskip
\emph{(1a)} \ First we show that $\mathcal{T}_m$ is continuous. In fact, let $\{ q_n\} \subset Q$ satisfying $q_n \to q$ in $C([a,b], E)$ and let $\{ \lambda_n\}\subset [0,1]$ with $\lambda_n \to \lambda$. For every $n \in \mathbb{N}$, the function $x_n :=\mathcal{T}_m(q_n, \lambda_n)$ is such that
\begin{equation}\label{e:xn}
x_n(t)=\left\{ \begin{array}{ll} \lambda_n S(\frac 1m )M(q_n), & t\in [a, a+\frac 1m]\\
y_{q_n,\lambda_n}(t)+\eta_{q_n,\lambda_n}(t) & t\in [a+\frac 1m, b]
\end{array}
\right.
\end{equation}
(see \eqref{e:yeta}). Put $x: =\mathcal{T}_m(q, \lambda)$. Since, by \emph{(A)}, $ \Vert S(t) \Vert \le 1 $ for $ t\ge 0, $  when $t\in [a, a+\frac 1m]$ we have that
\begin{equation*}
\begin{array}{rl}
\Vert x_n(t) -x(t) \Vert =&\Vert \lambda_n S(\frac 1m)M(q_n) -\lambda S(\frac 1m)M(q)\Vert \\
\le &\vert \lambda_n -\lambda\vert \Vert S(\frac 1m) M(q_n) \Vert +\lambda \Vert S(\frac 1m)\left[M(q_n) -M(q) \right] \Vert \\
\le & \vert \lambda_n -\lambda\vert \Vert M(q_n) \Vert +\lambda \Vert M(q_n) -M(q) \Vert .
\end{array}
\end{equation*}
Since $M(q_n) \to M(q)$ in $E$ (see Remark \ref{r:M}) hence, in particular, $\{M(q_n) \}$ is bounded in $E$, we obtain that $x_n \to x$ in $C([a, a+\frac 1m], E)$.\\
Similarly it is easy to prove that
\begin{equation}\label{e:yconv}
y_{q_n, \lambda_n}\to y_{q,\lambda} \enspace \text{in } C([a+\frac 1m, b], E).
\end{equation}
Notice that, by \emph{(f)} and the convergence of $ \{q_n\} $ to $ q, $ it follows that $f(t, q_n(t)) \to f(t, q(t)), \, t \in [a,b]. $ Again by \emph{(f)} the convergence is also dominated since \begin{equation}\label{e:fqn}
\Vert f(t, q_n(t)) \Vert \le \nu_{r\overline B}(t), \quad \text{for a.a. } t \in [a,b].
\end{equation}
By the Lebesgue dominated convergence theorem we conclude that, for every $ t \in [a+\frac 1 m,b], $
\begin{equation*}
\begin{array}{rl}
\Vert \eta_{q_n, \lambda_n}(t) -\eta_{q, \lambda}(t) \Vert \le&\vert \lambda_n - \lambda \vert \int_{a+\frac 1 m}^t \Vert S(t-s) f(s, q_n(s)) ds \Vert + \\
& \hskip .2 cm \lambda \int_{a+\frac 1 m}^t \Vert S(t-s) [f(s, q_n(s)) - f(s, q(s))] \Vert ds \\
\le & \vert \lambda_n -\lambda\vert  \int_{a+\frac 1m}^{b} \nu_{r\overline B}(s) \, ds + \lambda  \int_{a+\frac 1 m}^b \Vert f(s, q_n(s)) - f(s, q(s)) \Vert ds \to 0.
\end{array}
\end{equation*}
Hence $\eta_{q_n, \lambda_n}\to \eta_{q,\lambda}$ in $ C([a+\frac 1m, b], E) $ and by \eqref{e:yconv} this proves that $x_n \to x$ in $C([a,b], E)$, i.e. $\mathcal{T}_m$ is a continuous operator.

\bigskip
\emph{(1b)} \ Now we show that $\mathcal{T}_m$ is compact. To this aim consider a sequence $\{x_n\}\subset \mathcal{T}_m(Q\times [0,1])$ which implies the existence of $\{q_n\} \subset Q$ and $\{\lambda_n\}\subset [0, 1]$ such that $x_n=\mathcal{T}_m(q_n, \lambda_n)$, i.e. $x_n$ satisfies condition \eqref{e:xn}, for all $n \in \mathbb{N}$. With no loss of generality we can restrict to a subsequence, as usual denoted as the sequence, such that $\lambda_n \to \lambda \in [0,1]$.

\noindent First consider the interval $[a, a+\frac 1m]. $  Since, by \emph{(m$_0$)},  the sequence $\{M(q_n)\}\subset r\overline B$ is bounded and the semigroup $\{S(t)\}_{t\ge 0}$ is compact, we obtain that $x_n$ is relatively compact in $C([a, a+\frac 1m], E)$.

\noindent Let $t\in [a+\frac 1m, b]. $  As before the set $\{y_{q_n, \lambda_n}(t)\}\subset E$ is relatively compact in $E$. We prove now that $\{y_{q_n, \lambda_n}\}$ is equicontinuous in $ [a+\frac 1m, b]$. Let, in fact, $t \in [a+\frac 1m, b]$ and $\varepsilon >0$. The compactness of $\{S(t)\}_{t\ge 0}$ implies the continuity of $S \colon (0, \infty) \to \mathcal{L}(E) $ in the uniform operator topology (see \cite[Theorem 3.2 p. 49]{p}). Hence we can find $\delta=\delta(\varepsilon)>0$ satisfying
\begin{equation*}
\Vert S(t-a)-S(t^{\, \prime}-a)\Vert \le\frac{\varepsilon }{r}, \quad \vert t-t^{\, \prime}\vert <\delta \text{ with } t^{\, \prime}\in [a+\frac 1m, b].
\end{equation*}
 Therefore, by \emph{(m$_0$)}
\begin{equation*}
\Vert y_{q_n, \lambda_n}(t) -y_{q_n, \lambda_n}(t^{\, \prime})\Vert \le \lambda_n \Vert S(t-a)-S(t^{\, \prime}-a)\Vert \Vert M(q_n)\Vert \le\frac{\varepsilon }{r} \, \cdot \, r=\varepsilon,
\end{equation*}
for every $ n\in \mathbb{N} $ and $ t,t' \in [a+\frac 1m, b] $ with $ \vert t-t^{\, \prime}\vert <\delta. $\\
Then, by the abstract version of the Ascoli-Arzel\'a theorem, we obtain that
\begin{equation}\label{e:ycompact}
\{y_{q_n, \lambda_n}\} \text{ is relatively compact in } C([a+\frac 1m, b], E).
\end{equation}

\noindent Consider now the sequence $\{\eta_{q_n, \lambda_n}\}$ defined in \eqref{e:yeta}.
Fix $t\in [a+\frac 1m, b]$ and let $h_{n,t}(s):=S(t-s)f(s, q_n(s))$ for $n \in \mathbb{N}$ and $s\in [a+\frac 1m, t]$. According to \eqref{e:fqn} and since, by \emph{(A)}, $ \Vert S(t) \Vert \le 1 $ for $ t\ge 0, $ we have that $\Vert h_{n,t}(s)\Vert \le \nu_{r\overline B}(s)$ for a.a. $s\in [a+\frac 1m, t]$, hence $\{h_{n,t}\}$ is integrably bounded in $ [a+\frac 1m, t]$. Moreover, condition \eqref{e:fqn} implies that $\{f(s, q_n(s))\}$ is bounded in $E$ for a.a. $s\in [a+\frac 1m, t]$ and by the compactness of the semigroup $\{S(t)\}_{t\ge 0}$ we obtain that the sequence $\{h_{n,t}(s)\}$ is relatively compact for a.a. $s \in [a+\frac 1m, t]$. In conclusion the sequence $\{h_{n,t}\}$ is semicompact in $ [a+\frac 1m, t]$ (see Definition \ref{d:semi}).\\

\noindent Let us introduce now the operator $\hat S \colon L^1( [a+\frac 1m, t], E) \to C( [a+\frac 1m, t], E)$ given by
\begin{equation*}
\hat S \varphi (\tau):= \int_{a+\frac 1m}^{\tau} \varphi(s)\, ds.
\end{equation*}
It is easy to see that $\hat S$ satisfies both conditions \emph{(i)} and \emph{(ii)} of Theorem \ref{t:semicomp}; indeed $\hat S$ is the Cauchy operator (see Example \ref{ex:CO}) in the special case when the semigroup is identically equal to $I$. By virtue of Theorem \ref{t:semicomp}, the sequence $\{\eta_{q_n, \lambda_n}\}$   is relatively compact in $C([a+\frac 1m, t], E). $ Hence, in particular, $\{\eta_{q_n, \lambda_n}(t)\}$ is relatively compact in $E$ for all $t \in [a+\frac 1m, b]$.\\

\noindent We prove now that $\{\eta_{q_n, \lambda_n}\}$ is equicontinuous. In fact, fix $t \in [a+\frac 1m, b]$ and let $t^{\, \prime}\in [a+\frac 1m, b]$ with $t^{\, \prime}>t$. Notice that
\begin{equation*}
\Vert \eta_{q_n, \lambda_n}(t^{\, \prime}) -\eta_{q_n, \lambda_n}(t) \Vert = \lambda_n\Vert\int_{a+\frac 1m}^{t^{\, \prime}} S(t^{\, \prime}-s)f(s, q_n(s))\, ds-\int_{a+\frac 1m}^{t} S(t-s)f(s, q_n(s))\, ds\Vert.
\end{equation*}
For every given $\sigma \in (0, t-\frac 1m-a)$, we can then estimate  $\Vert \eta_{q_n, \lambda_n}(t^{\, \prime}) -\eta_{q_n, \lambda_n}(t) \Vert$ by means of the sum of the following three integrals
\begin{equation}\label{e:zn}
\begin{array}{rl}
\Vert \eta_{q_n, \lambda_n}(t^{\, \prime}) -\eta_{q_n, \lambda_n}(t) \Vert\le & \lambda_n \Vert \int_{a+\frac 1m}^{t-\sigma} [S(t^{\, \prime}-s)-S(t-s)]f(s, q_n(s))\, ds\Vert \\
\\
+& \lambda_n \Vert \int_{t-\sigma}^{t} [S(t^{\, \prime}-s)-S(t-s)]f(s, q_n(s))\, ds\Vert  \\
\\
+ & \lambda_n \Vert \int_{t}^{t^{\, \prime}} S(t^{\, \prime}-s)f(s, q_n(s))\, ds\Vert.
\end{array}
\end{equation}
Now fix $\varepsilon>0. $ With no loss of generality we can take $\sigma >0$ such that
\begin{equation}\label{e:6D}
\int_{t-\sigma}^t  \nu_{r \overline B}(s) \, ds <\frac{\varepsilon}{6}.
\end{equation}
Let us start from the first integral in \eqref{e:zn}. Notice that $t^{\, \prime}-s, t-s \in [\sigma, t^{\, \prime}-\frac 1m-a]$ for  $s \in [a+\frac 1m, t-\sigma]$. Since $\{S(t)\}_{t\ge 0}$ is compact, it is also uniformly continuous in the compact interval $ [\sigma, t^{\, \prime}-\frac 1m-a]$.  We can then find $\sigma_1(\varepsilon)>0$ such that
\begin{equation*}
\Vert S(\tau^{\, \prime})-S(\tau)\Vert \le \frac{\varepsilon}{3\Vert \nu_{r \overline B}\Vert}, \quad \vert \tau^{\, \prime}-\tau\vert <\sigma_1, \quad \tau^{\, \prime}, \tau \in \left[\sigma, t^{\, \prime}-\frac 1m -a\right].
\end{equation*}
Since $t^{\, \prime}-s-(t-s)=t^{\, \prime}-t$ for $s\in [a+\frac 1m, t-\sigma]$, when $0<t^{\prime}-t <\sigma_1, $ by \eqref{e:fqn} and since $ \Vert S(t) \Vert \le 1 $ for $ t\ge 0, $ we have
\begin{equation*}
\begin{array}{rl}
\lambda_n \Vert \int_{a+\frac 1m}^{t-\sigma} [S(t^{\, \prime}-s)-S(t-s)]f(s, q_n(s))\, ds\Vert \le& \int_{a+\frac 1m}^{t-\sigma} \Vert S(t^{\, \prime}-s)-S(t-s)\Vert \Vert f(s, q_n(s))\Vert \, ds \\
\le & \frac{\varepsilon}{3\Vert \nu_{r \overline B}\Vert }\int_a^b \Vert f(s, q_n(s)) \Vert \, ds \le \frac{\varepsilon}{3}.
\end{array}
\end{equation*}
According to \eqref{e:6D} the second integral in \eqref{e:zn} is such that
\begin{equation*}
\begin{array}{rl}
\lambda_n\Vert \int_{t-\sigma}^t [S(t^{\, \prime}-s)-S(t-s)]f(s, q_n(s))\, ds\Vert  \le &  \int_{t-\sigma}^t [\Vert S(t^{\, \prime}-s)\Vert + \Vert S(t-s)\Vert] \Vert f(s, q_n(s))\Vert \, ds\\
\le & 2\int_{t-\sigma}^t \nu_{r\overline B}(s) \, ds\le \frac{\varepsilon}{3}.
\end{array}
\end{equation*}
Moreover, let  $\sigma_2>0$ be such that
\begin{equation*}
\int_{t}^{t^{\, \prime}}\nu_{r\overline B}(s) \, ds \le \frac{\varepsilon}{3}, \quad \text{for }  t^{\, \prime}-t<\sigma_2;
\end{equation*}
hence we obtain the following estimate for the third integral in \eqref{e:zn}
\begin{equation*}
\lambda_n\Vert \int_{t}^{t^{\, \prime}} S(t^{\, \prime}-s)f(s, q_n(s)) \, ds \Vert \le \int_{t}^{t^{\, \prime}}\nu_{r\overline B}(s) \, ds \le \frac{\varepsilon}{3}.
\end{equation*}
Consequently, when $t^{\, \prime}-t <\min\{\sigma_1, \sigma_2\}$, we have that $\Vert \eta_{q_n, \lambda_n}(t^{\, \prime})-\eta_{q_n, \lambda_n}(t)\Vert <\varepsilon; $ since the reasoning is similar also in the case $t^{\, \prime}\in [a+\frac 1m, t]$ we conclude that $\{\eta_{q_n, \lambda_n}\}$ is equicontinuous in $[a+\frac 1m, b]$ and again we can use the abstract version of Arzel\`a-Ascoli theorem in order to show that  $\{\eta_{q_n, \lambda_n}\}$ is relatively compact in $C([a+\frac 1m, b], E)$. Therefore, by \eqref{e:ycompact} and the estimates in the interval $[a, a+\frac 1m]$ (see the beginning of (1b)),  we have that $\{x_n\}$ is relatively compact and so the operator $\mathcal{T}_m$ is compact. In conclusion   $\mathcal{T}_m$ is completely continuous since it is both continuous and compact.

\bigskip

\emph{(1c)} \ For every $q\in Q$ we have that $\mathcal{T}_m(q, 0)\equiv 0$ and since $0\in rB$, it implies that $\mathcal{T}_m(Q, 0) \subset \mbox{int} \, Q$.

\bigskip
\emph{(1d)} \ We apply now a degree argument for the study of the fixed points of $\mathcal{T}_m(\cdot, 1) $ and then we need to show that $\mathcal{T}_m(\cdot, \lambda)$ is fixed points free on $\partial Q$ for every $\lambda \in [0,1]$. The case $\lambda=0$ was already treated in \emph{(1c)}. Any possible fixed point $x\in \partial Q$ for $\lambda=1$, i.e. satisfying $x = \mathcal{T}_m(x, 1)$, is already a solution of our problem. So, it remains to show that $\mathcal{T}_m(\cdot, \lambda)$ is fixed-points free on $\partial Q$ only for $\lambda \in (0,1)$. We reason by  contradiction and assume the existence of $(x, \lambda) \in \partial Q \times (0,1)$ satisfying $x=\mathcal{T}_m(x, \lambda)$. According to the definition of $Q$, there exists $t_0 \in [a,b]$ such that $\Vert x(t_0) \Vert=r. $ By \emph{(m$_0$)} and since $\{S(t)\}_{t\ge 0} $ is a semigroup of contractions, the case $t_0 \in [a, a+\frac 1m]$ leads to the contradictory conclusion
\begin{equation*}
r=\Vert x(t_0) \Vert= \lambda \Vert S(\frac 1m)M(x)\Vert \le \lambda r <r.
\end{equation*}
Consequently $t_0 \in (a+\frac 1m, b]; $ we recall that $x(t)$ is a mild solution of the equation
\begin{equation}\label{e:g0}
x^{\prime}(t)=Ax(t)+h(t,x(t)), \quad t\in [a+\frac 1m, b]
\end{equation}
where
\begin{equation}\label{e:flambda}
h(t,x):=\lambda f(t,x), \quad (t,x)\in [a,b]\times E.
\end{equation}
Notice that $ h $ satisfies conditions \emph{(f)}.  We show that $ V $ is a bounding function for \eqref{e:g0} with $ K=rB, $ in particular that $ V $ satisfies condition \emph{(V2)}. In fact, let $x \in E$ with $\Vert x \Vert =r $ and $t \in (a, b]$; since $ V $ is a bounding function for \eqref{e:equation} there is a sequence $\{h_n\}\subset \mathbb{R}$, with $h_n \to 0^-$ such that
\begin{equation*}
\lim_{n\to \infty}\frac{V(x+h_nf(t,x))}{h_n}<0.
\end{equation*}
Put $k_n:=\frac{ h_n}{\lambda}, \, n \in \mathbb{N}$; we have that $k_n \to 0^-$ and
\begin{equation*}
\lim_{n\to \infty}\frac{V(x+ k_ng(t,x))}{k_n}=\lim_{n\to \infty}\frac{V(x+ k_n \lambda f(t,x))}{k_n}=\lambda \, \lim_{n\to \infty}\frac{V(x+ h_nf(t,x))}{h_n}<0.
\end{equation*}
 Hence $ V $ is a bounding function for \eqref{e:g0} and by applying Theorem \ref{t:bf} to \eqref{e:g0} we obtain that $\Vert x(t_0) \Vert <r. $ In conclusion $\Vert x(t) \Vert <r $ for all $t \in [a+\frac 1m, b] $ and hence $\mathcal{T}_m$ is fixed points free on the boundary $\partial Q$.

\bigskip
\emph{(1e)} \ The properties already proved imply that $\mathcal{T}_m$ is an admissible homotopy which connects the fields $\mathcal{T}_m(\cdot, 0)$ and $\mathcal{T}_m(\cdot, 1)$. According to the homotopic invariance and the normalization property of the Leray-Schauder topological degree, we then obtain
\begin{equation*}
deg( i-\mathcal{T}_m(\cdot, 1), \, Q)=deg( i-\mathcal{T}_m(\cdot, 0), \, Q)=1
\end{equation*}
and hence there exists $x_m \in Q$ such that $x_m=\mathcal{T}_m(x_m, 1)$, i.e. $x_m(t) \equiv S(\frac 1m)M(x_m) $ for $t \in[a, a+\frac 1m]$ and $x_m$ is a solution of the initial value problem \eqref{e:x}-\eqref{e:x(a+1/m)} for $\lambda=1$ on $[a+\frac 1m, b]; $ moreover $\Vert x_m(t) \Vert \le r$ for $t \in [a,b]$ with $r$ introduced in \emph{(m$_0$)}.

\bigskip
\noindent STEP 2. \ In this part we consider the sequence of functions $\{x_m\}$ obtained in the previous step and, by passing to the limit, we get a solution of problem \eqref{e:equation}-\eqref{e:nonlocal}. According to \eqref{e:DTm}, we recall that
\begin{equation}\label{e:xm}
x_m(t) =\left\{\begin{array}{rl}
S(\frac 1m)M(x_m) \quad& t\in [a, a+\frac 1m]\\
S(t-a)M(x_m)+\int_{a+\frac{1}{m}}^t S(t-s)f(s, x_m(s)) \, ds, \quad& t \in [a+ \frac 1m, b].
\end{array}
\right.
\end{equation}

\smallskip
\emph{(2a)} \ Take $\alpha \in (a,b]$ and let $m$ be sufficiently large so that $a+\frac 1m <\alpha. $ Since $\{x_m\}\subset Q$, according to \emph{(f)} we have that $\Vert f(t, x_m(t)) \Vert \le\nu_{ r\overline B}(t)$ for a.a. $t \in [a,b]$. Hence, with a similar reasoning as in \emph{(1b)}, we can show that $\{x_m\}$ is relatively compact in $[\alpha, b]$.

\smallskip
\emph{(2b)} \ Fix a decreasing sequence $\{ a_n \}\subset (a,b)$ satisfying $a_n \to a$ as $n \to \infty$. According to (2a), $ \{ x_m\} $ is relatively compact in $ C([a_1,b],E); $ hence there is a subsequence $ \{ x_m^{(1)}\} $ converging in $ C([a_1,b],E) $ to a continuous function $ \overline x:[a_1,b] \to E. $ Similarly, there exists a subsequence $ \{ x_m^{(2)} \} $ of $\{ x_{m}^{(1)} \}$ converging in $ C([a_2,b],E) $ to a continuous function $ \overline {\overline x}: [a_2,b] \to E $ and, according to the unicity of the limit, $ \overline x(t) = \overline {\overline x}(t) $ for $ t \in [a_1,b]. $  Proceeding by induction, for every $n \in \mathbb{N}$ we can find a sequence $\{ x_{m}^{(n)} \}$ which is a subsequence of $\{ x_{m}^{(n-1)} \}$ and converges in $C([a_n, b], E)$. According to the unicity of the limit, we can define a continuous function $ \tilde x:(a,b] \to E$ and, using a Cantor diagonalization argument, the sequence $ x_n^{(n)} (t) \to \tilde x(t)$ for $ t\in (a, b]$. By the continuity of $ f $ (see condition \emph{(f)}), it implies that $f(t, x_{n}^{(n)} (t)) \to f(t, \tilde x(t))$ for all $t \in (a,b]$ and also that $t \longmapsto f(t, \tilde x(t)) $ is continuous on $(a, b]. $  Moreover, since $\{ x_m \} \subset Q$, again by \emph{(f)} we have that $\Vert f(t, x_{n}^{(n)}(t)) \Vert \le \nu_{r\overline B}(t)$ and hence $\Vert f(t, \tilde x(t)) \Vert \le \nu_{r\overline B}(t)$ for a.a. $t\in (a,b]$, with $\nu_{r\overline B} \in L^1([a,b])$.

\smallskip
\emph{(2c)} \ Since $\{ x_{n}^{(n)}\} \subset C([a,b], r\overline B), $ by \emph{(m$_0$)} we obtain that $\{M(x_{n}^{(n)})\}\subset r\overline B. $ Hence, by the reflexivity of $ E, $ there is a subsequence, as usual denoted as the sequence, satisfying
\begin{equation}\label{e:x0}
M\left(x_{n}^{(n)}\right)\rightharpoonup x_0 \enspace  \mbox{as } n \to \infty, \enspace \text{with } \Vert x_0 \Vert \le r.
\end{equation}

\smallskip
\emph{(2d)} \ Let us introduce, now, the continuous function
\begin{equation}\label{e:solution}
x(t):=S(t-a)x_0+\int_a^t S(t-s)f(s, \tilde x(s)) \, ds, \qquad t \in [a,b],
\end{equation}
with $x_0$ as in \emph{(2c)} and $\tilde x$ defined in \emph{(2b)}. The integral in \eqref{e:solution} is well defined by the regularity of the function $t \longmapsto f(t, \tilde x(t)) $ in $[a, b] $ showed in \emph{(2b)}. We claim that
\begin{equation}\label{e:conv}
x_{n}^{(n)}(t) \to x(t), \enspace t \in (a,b], \qquad \mbox{as } n \to \infty.
\end{equation}
In fact, by \eqref{e:xm} and the definition of $ \{ x_{n}^{(n)}\}, $ there is a sequence $ \{p_n\}, $ with $ 0<p_n \le 1/n, \, n \in \mathbb{N}, $ such that
\begin{equation}\label{e:xpn}
x_{n}^{(n)}(t)=\left\{
\begin{array}{ll} S(p_n)M(x_{n}^{(n)}) & t\in [a, a+p_n]\\
S(t-a)M(x_{n}^{(n)})+\int_{a+p_n}^t S(t-s)f(s,x_{n}^{(n)}(s))ds & t\in[a+p_n, b].
\end{array}
\right.
\end{equation}
Let $t \in (a,b]. $ By \eqref{e:x0} we have that
\begin{equation*}
S(t-a)M(x_{n}^{(n)})\rightharpoonup S(t-a)x_0.
\end{equation*}
Now take $n \ge \overline n $ such that $ t \in [a+p_n, b] $ for every $ n \ge \overline n. $ Notice that
\begin{equation*}
\int_{a+p_n }^t S(t-s)f(s, x_{n}^{(n)}(s)) \, ds= \int_a^t S(t-s)f(s, x_{n}^{(n)}(s)) \, ds-\int_a^{a+p_n } S(t-s)f(s, x_{n}^{(n)}(s)) \, ds.
\end{equation*}
By the properties showed in \emph{(2b)}, we have
\begin{equation*}
S(t-s)f(s, x_{n}^{(n)}(s)) \to S(t-s)f(s, \tilde x(s)), \enspace  s \in (a, t].
\end{equation*}
Moreover $\Vert S(t-s)f(s, x_{n}^{(n)}(s)) \Vert \le \nu_{r\overline B}(s)$  a.e. in $[a,t], $ with $\nu_{r\overline B} \in L^1([a, t])$. So the Lebesgue Dominated Convergence Theorem leads to
\begin{equation*}
\int_a^t S(t-s)f(s, x_{n}^{(n)}(s)) \, ds \to \int_a^t S(t-s)f(s, \tilde x(s)) \, ds,  \enspace \text{as } n \to \infty.
\end{equation*}
Since $0<p_n \le 1/n, \, n \in \mathbb{N}, $ we have
\begin{equation*}
\int_a^{a+p_n} S(t-s)f(s, x_{n}^{(n)}(s)) \, ds \to 0 \enspace \text{as } n \to \infty;
\end{equation*}
 hence, by \eqref{e:xpn} , $x_{n}^{(n)}(t) \rightharpoonup x(t). $  Since we showed in \emph{(2b)} that $x_{n}^{(n)}(t) \to \tilde x(t)$ for $t \in (a, b]$, we obtain that $\tilde x(t)=x(t)$ and hence, by \eqref{e:solution}, $x$ is a mild solution of \eqref{e:equation} in $[a,b]$. By condition \emph{(m$_1$)} we get that $M(x_{n}^{(n)}) \to M(x) $ thus \eqref{e:x0} implies that $ x(a)=x_0=M(x) $ and hence $x$ satisfies the boundary condition \eqref{e:nonlocal}. The proof is complete.
 \end{proof}

\section{Nonlocal solutions of the parabolic equation \eqref{e:PE}} \label{s:application}
In this part we apply the methods and results contained in Section \ref{s:abstrat} for the study of nonlocal boundary value problems associated to equation \eqref{e:PE}. We prove, in particular, Theorem \ref{t:nonPEP} as a special case of the following more general result (see Theorem \ref{t:PEgen}).\\
We put
\begin{equation} \label{e:maxg}
\mu:= \max_{ t\in [0, T]} \vert g(t,0)\vert.
\end{equation}
The definition is well posed by the continuity of $g$ (see condition \eqref{e:gandk}\emph{(i)}).
\begin{theorem}\label{t:PEgen}
Consider problem \eqref{e:PE}-\eqref{e:D}-\eqref{e:nonPE}. Let $a_{i,j}\in C^1(\overline D), \, i,j=1,...,n$ satisfy conditions \eqref{e:sym}, \eqref{e:elliptic} and  assume \eqref{e:gandk}; suppose \emph{($m_0$)} and \emph{($m_1$)} hold (see Section \ref{s:abstrat}) with $E=L^p(D)$. If, moreover,
\begin{itemize}
\item[(a)] $b>L+\vert D\vert $

\item [(b)] $\displaystyle{r>\frac{\left( L+\mu\right)\vert D\vert}{b-\vert D\vert-L}}$,
\end{itemize}
with $\mu$ as in \eqref{e:maxg}, then \eqref{e:PE}-\eqref{e:D}-\eqref{e:nonPE} is solvable.

\end{theorem}
\begin{proof} \ We reformulate \eqref{e:PE}-\eqref{e:D}-\eqref{e:nonPE} in the  abstract setting \eqref{e:equation}-\eqref{e:nonlocal} where $ E:=L^p(D), \, 1<p<\infty. $ \\ The linear elliptic partial differential operator in divergence form $A_p$ introduced in \eqref{e:divergence} generates a compact $C_0$-semigroup of contractions $\{S(t)\}_{t\ge 0} $ (see Theorem \ref{t:compact}).  \\
 Given $ \eta \in E $ we denote, as usual, with $ \vert \eta\vert $ the map $ \xi \longmapsto \vert \eta(\xi)\vert $ for a.a. $ \xi \in D. $ Consider  $\beta \in (0,1); $ since $ D $ is bounded and  $ c>0 $ implies  $ \displaystyle{c^{\beta}\le \max\{ 1, \, c\}}, $ it is clear that also $\vert \eta \vert^{\beta} \in E. $ Moreover it is easy to see that $ \vert \eta \vert^{p\beta} \in L^{1/\beta} (D) $ and the following estimate is satisfied
\begin{equation}\label{e:eta}
\left \Vert \vert \eta \vert^{\beta} \right \Vert \le \vert D \vert^{\frac{1-\beta}{p}} \cdot \Vert \eta \Vert^{\beta}, \enspace \eta \in E.
\end{equation}

\smallskip
\noindent We introduce the function $ f \colon [0, T]\times E \to E$ defined by
\begin{equation}\label{e:fH}
f(t,\eta)(\xi):=\int_D k(\xi,y) \eta(y) \, dy-b\eta(\xi)+g(t, \eta(\xi)), \enspace \text{for a.a. } \xi \in D.
\end{equation}
By \eqref{e:gandk} we have
\begin{equation*}
\vert f(t,\eta)(\xi)\vert \le \int_{D} \vert \eta(y)\vert \, dy +b\vert \eta(\xi)\vert + \vert g(t, \eta(\xi))-g(t,0)\vert +\vert g(t,0)\vert,
\end{equation*}
for $ t \in [0,T] $ and a.a. $ \xi \in D. $
Hence, by  the H\"{o}lder inequality, we obtain
\begin{equation}\label{e:about f}
\begin{array}{rl}\vert f(t,\eta)(\xi) \vert\le &\vert D\vert^{1-\frac 1p} \Vert \eta \Vert +b\vert \eta(\xi)\vert +L\max\{\vert \eta(\xi)\vert, \, \vert \eta(\xi)\vert^{\beta}\}+ \mu\\
&\vert D\vert^{1-\frac 1p} \Vert \eta \Vert+(b+L)\vert \eta(\xi)\vert +L+\mu, \enspace \text{for a.a. } \xi \in D
\end{array}
\end{equation}
with $\mu$ as in \eqref{e:maxg}, implying that $f$ is well-defined. Therefore, in abstract setting, equation \eqref{e:PE} takes the form \eqref{e:equation} with $x(t)=u(t, \cdot). $\\
Now we show that $f$ is continuous. Let $(t_n, \eta_n) \to (t,\eta)$ in $[0, T]\times E. $
By the continuity of $g$ we have that $g(t_n, \eta(\xi)) \to g(t, \eta(\xi)), $ for a.a. $ \xi \in D$ and the convergence is dominated in $ E $ since
\begin{equation*}
\begin{array}{rl}
\vert g(t_n, \eta(\xi))-g(t, \eta(\xi))\vert \le & \vert g(t_n, \eta(\xi))-g(t_n, 0)\vert + \vert g(t, \eta(\xi))-g(t, 0)\vert\\
& + \, \vert g(t_n, 0)\vert + \vert g(t, 0)\vert \\
\le  & 2L(1+\vert \eta(\xi)\vert)+2\mu, \enspace \text{for a.a. } \xi \in D.
\end{array}
\end{equation*}
Therefore
\begin{equation}\label{e:g}
\Vert g(t_n, \eta(\cdot)) -g(t, \eta(\cdot)) \Vert \to 0, \enspace \text{as }n \to \infty, \text{ in } E.
\end{equation}
By \eqref{e:gandk} and the H\"{o}lder inequality, we have
\begin{equation*}
\begin{array}{rl}
\vert  f(t_n, \eta_n)(\xi)-f(t, \eta)(\xi)\vert \le & \vert  f(t_n, \eta_n)(\xi)-f(t_n, \eta)(\xi)\vert+
\vert  f(t_n, \eta)-f(t, \eta)\vert(\xi)\\
\le &\vert D \vert^{1-\frac1 p}\Vert \eta_n -\eta \Vert +(b+L) \vert \eta_n(\xi)-\eta(\xi)\vert + L\vert \eta_n(\xi)-\eta(\xi) \vert^{\beta}\\
& +\vert g(t_n, \eta(\xi))-g(t, \eta(\xi))\vert.
\end{array}
\end{equation*}
Hence, by \eqref{e:eta} and \eqref{e:g}, $f(t_n, \eta_n) \to f(t,\eta)$ in $ E $ and $ f $ is continuous.\\
We prove that $ f $ satisfies also the  growth condition in \emph{(f)}. So, let $\Omega \subset E$ be bounded and take $\eta \in \Omega. $ By the estimate
\begin{equation}\label{e:abp}
(a+b)^p \le 2^p (a^p + b^p) \quad \text{for } a,b \ge 0 \text{ and } p>1,
\end{equation}
and according to \eqref{e:about f}  we have
\begin{equation*}
\Vert  f(t, \eta)\Vert^p = \int_{D}\left[ \vert f(t, \eta)\vert(\xi)\right]^p \, d\xi\le  2^p(b+L)^p\Vert \eta \Vert^p +2^p \left( \vert D \vert^{1-\frac1p}\Vert \eta \Vert +L+\mu\right)^p \vert D \vert,
\end{equation*}
for all $t \in [0,T] $ and hence \emph{(f)} is satisfied.

\smallskip
 By assumption, the nonlocal condition \eqref{e:nonPE} satisfies both \emph{(m$_0$)} and \emph{(m$_1$)}.

\smallskip
It remains to show the existence of a locally Lipschitzian bounding function $V \colon E \to \mathbb{R}$ (Definition \ref{d:bf}) with $K=rB$ and $r$ as in \emph{($m_0$)}. Consider the function
\begin{equation*}
V_r(x)=\frac 12 \left(\Vert x\Vert^2-r^2 \right), \quad x\in E.
\end{equation*}
$V_r$ is locally Lipschitzian; it is  also Fr\'{e}chet  differentiable, since $L^p(D)$ is uniformly convex, and
$$
\langle \dot V_r(y), z \rangle =  \displaystyle\frac{1}{\|y\|^{p-2}} \displaystyle\int_D |y(\xi)|^{p-2} \; y(\xi) \; z(\xi) \, d\xi, \enspace y, z \in E
$$
(see e.g. Example \ref{ex:Vnorma}\emph{(ii)}). We prove that $V_r$ satisfies condition \emph{(V2)}. Notice that, since $D$ is bounded,
\begin{equation*}
\int_D\vert \eta(\xi)\vert^{p-1}\, d\xi \le \vert D \vert^{\frac 1p}\Vert \eta\Vert ^{p-1}, \enspace \eta \in L^p(D).
\end{equation*}
Therefore, we have the following estimate
\begin{equation*}
\begin{array}{rl}
&\left \vert \int_{D} |\eta(\xi)|^{p-2}\eta(\xi) \left [ \int_D k(\xi, y)\eta(y)\, dy +g(t,\eta(\xi))\right]\, d\xi\right \vert\\
\\
\le &\int_{D} |\eta(\xi)|^{p-1}\left(  \int_D \vert \eta(y)\vert\, dy +\vert g(t,\eta(\xi))-g(t,0)\vert +\mu\right)\, d\xi\\
\\
\le &\int_{D} |\eta(\xi)|^{p-1}\left( L\vert \eta(\xi)\vert +\vert D\vert^{1-\frac 1p}\Vert \eta\Vert +L+\mu\right)\, d\xi\\
\\
\le &L\int_D\vert \eta(\xi)\vert^p \, d\xi + \left(\vert D\vert^{1-\frac 1p}\Vert \eta \Vert +L+\mu \right)\int_D\vert \eta(\xi)\vert^{p-1}\, d\xi\\
\\
\le &L\Vert \eta\Vert^{p}+\left(\vert D\vert^{1-\frac 1p}\Vert \eta \Vert +L+\mu \right)\vert D\vert^{\frac 1p}\Vert \eta\Vert ^{p-1}\\
\\
=& \left(\vert D\vert +L \right)\Vert \eta\Vert^{p}+\left( L+\mu\right)\vert D\vert^{\frac 1p}\Vert \eta\Vert^{p-1}.
\end{array}
\end{equation*}
Thus, if $\Vert \eta \Vert =r$, by means of conditions \emph{(a)}-\emph{(b)} and the H\"{o}lder inequality, we have that
\begin{equation*}
\begin{array}{rl}
\langle \dot V_r(\eta),  f(t,\eta)\rangle =& \displaystyle \frac{1}{\|\eta\|^{p-2}}\int_{D} |\eta(\xi)|^{p-2}\eta(\xi) \left [ \int_D k(\xi, y)\eta(y)\, dy -b\eta(\xi)+g(t,\eta(\xi))\right]\, d\xi\\
\\
=&-b \Vert \eta \Vert^2+\frac{1}{\|\eta\|^{p-2}} \int_{D} |\eta(\xi)|^{p-2}\eta(\xi) \left [ \int_D k(\xi, y)\eta(y)\, dy +g(t,\eta(\xi))\right]\, d\xi\\
\\
\le&-b \Vert \eta \Vert^2+\frac{1}{\|\eta\|^{p-2}} \left \vert \int_{D} |\eta(\xi)|^{p-2}\eta(\xi) \left [ \int_D k(\xi, y)\eta(y)\, dy +g(t,\eta(\xi))\right]\, d\xi\right \vert\\
\\
\le &\left(-b+\vert D\vert +L \right)\Vert \eta\Vert^{2}+\left( L+\mu\right)\vert D\vert^{\frac 1p}\Vert \eta\Vert \\
\\

= & (-b+\vert D\vert +L)r^2+ \left( L+\mu\right)\vert D\vert^{\frac 1p}\, r.
\end{array}
\end{equation*}
Since $b>L+\vert D\vert$, when
\begin{equation*}
r>\frac{(L+\mu) \vert D \vert^{1/p}}{b-\vert D\vert -L}
\end{equation*}
we obtain that
\begin{equation*}
\langle \dot V_r(\eta), \, f(t, \eta)\rangle <0, \enspace \Vert \eta \Vert =r
\end{equation*}
and then $V_r$ is a locally Lipschitzian bounding function for \eqref{e:equation} (see Remark \ref{r:bf}).

\smallskip
\noindent All the assumptions of Theorem \ref{t:main} are then satisfied and the proof is complete
\end{proof}

\bigskip

\noindent \emph{Proof of Theorem \ref{t:nonPEP} }\ The proof follows from Theorem \ref{t:PEgen}. It remains only to show that \eqref{e:mean} and \eqref{e:Cauchy} in the abstract setting  satisfy \emph{($m_0$)} and \emph{($m_1$)} in $E=L^p(D)$.
\begin{itemize}
\item[\emph{(i)}] \ Let $M \colon C([0,T], L^p(D)) \to L^p(D)$ be defined by
\begin{equation*}
Mx=\frac 1T \int_0^T x(t)\, dt.
\end{equation*}
Notice that the definition is well-posed since $x(t)$ is a continuous function. Let $r>0$ and consider  $x \in C([0,T], L^p(D))$ with  $\Vert x\Vert \le r $. Then
\begin{equation*}
\Vert Mx \Vert \le \frac 1T \int_0^T \Vert x(t) \Vert \, dt \le r;
\end{equation*}
hence condition \emph{($m_0$)} is satisfied for any $r>0$. With no loss of generality, we can then assume that also condition \emph{(b)} in Theorem \ref{t:PEgen} is satisfied.  Let $\{x_n\}\subset C([0,T], L^p(D)) $ be such that $x_n(t) \to x(t), \, t \in (0,T]$ with $x \in C([0,T], L^p(D))$ and $\Vert x_n\Vert\le r$ for all $n$. The  convergence of $\{x_n\}$ is then dominated, implying that
\begin{equation*}
Mx_n=\frac 1T \int_0^T x_n(t)\, dt \to \frac 1T \int_0^T x(t)\, dt=Mx;
\end{equation*}
hence also \emph{($m_1$)} is satisfied. By applying Theorem \ref{t:PEgen}, we state claim \emph{(i)}.
\item[\emph{(ii)}] Let $M \colon C([0,T], L^p(D)) \to L^p(D)$ be such that
\begin{equation*}
Mx=\sum_{i=1}^q \alpha_i x(t_i), \enspace \text{with } t_i\in (0,T], \, \alpha_i\in \mathbb{R}, i=1,...,q \text{ and } \sum_{i=1}^q \vert \alpha_i\vert \le 1.
\end{equation*}
If $x \in C([0,T], L^p(D))$ with $\Vert x\Vert \le r, \, r>0$ we have
\begin{equation*}
\Vert Mx\Vert \le \sum_{i=1}^q \vert \alpha_i\vert \Vert x(t_i)\Vert \le r  \sum_{i=1}^q \vert \alpha_i\vert\le r,
\end{equation*}
implying condition \emph{($m_0$)}. As in \emph{(i)}, by the arbitrariness of $r$ we can assume that condition \emph{(b)} in Theorem \ref{t:PEgen} is satisfied. If, moreover,  $\{x_n\}\subset C([0,T], L^p(D)) $ and $x \in C([0,T], L^p(D)) $ are defined as in \emph{(i)}, it is easy to see that
\begin{equation*}
Mx_n=\sum_{i=1}^q \alpha_i x_n(t_i) \to \sum_{i=1}^q \alpha_i x(t_i) =Mx, \text{ as } n \to \infty,
\end{equation*}
so also \emph{($m_1$)} is satisfied. Claim \emph{(ii)} then follows, again by Theorem \ref{t:PEgen}, and the proof is complete.
\end{itemize}

\section{Bounding functions for mild solutions}\label{s:bounding}
This part contains a brief discussion about the notion of bounding function introduced in Section \ref{s:abstrat} (see Definition \ref{d:bf}).
Notice that the set $ \overline K $ in Definition \ref{d:bf} is  the $0$-sublevel set of $ V. $ The function $ V $ takes its name from its relevant property. In fact, when such a $ V $ exists,  every solution $ x \in C([a,b], E) $ of \eqref{e:equation} which is located in $ \overline K $ lies, indeed, in $ K $  for $ t \in (a,b] $ (see Theorem \ref{t:bf}). Hence, the existence of a bounding function for \eqref{e:equation} makes the transversality condition automatically satisfied on $ (a,b] $ and it remains to check the behaviour of the solution only for  $ t = a. $ \\
The bounding function theory was originally introduced in \cite{GaMa1} and \cite{MT} (see also \cite{GaMa2}), in the framework of finite dimensional systems and with smooth  bounding function. Again in Euclidean spaces, the theory was extended in \cite{T} and in \cite{Za}, to the case of non-smooth functions. The bounding function theory was developed in \cite{AMT}, in an infinite dimensional setting, when $A\colon E \to E$ is linear and bounded and then \eqref{e:equation} has classical, i.e. absolutely continuous, solutions. A special type of bounding function (see \eqref{e:Vr} below) was used in \cite{BLT}, in combination with the Yoshida approximation of the linear part.

\smallskip
To the best of our knowledge, no result about the existence of bounding functions is available in the present general framework, i.e. in an arbitrary Banach space when the linear term is not necessarily bounded and it generates a $C_0$-semigroup.

\smallskip

\begin{remark}\label{r:bf}Let us denote with $\langle \cdot, \, \cdot \rangle$ the duality between $E$ and its dual space $E^*$. When $V$ is G\^{a}teaux differentiable  it is easy to see that
\begin{equation*}
\lim_{h \to 0}\frac{V(x+hy)-V(x)}{h}=\langle \dot V(x), \, y \rangle, \quad x, y \in E.
\end{equation*}
Therefore, since $V(x)=0, \, x \in \partial K, $ if  $V$ is G\^{a}teaux differentiable on $\partial K, $  condition \emph{(V2)} simply reduces to the inequality
\begin{equation}\label{e:CH}
\langle \dot V(x), \, f(t,x) \rangle<0, \enspace \text{for } t \in (a,b] \text{ and } x \in \partial K.
\end{equation}
\end{remark}
\begin{theorem} \label{t:bf} \ Let $ E $ be a Banach space, $ A: D(A) \subset E \rightarrow E $
linear, not necessarily bounded and such that it generates a $C_0$-semigroup $S: [0,\infty) \to \mathcal{L}(E)$ and  $f\colon [a,b]\times E \to E$ continuous. Assume the existence of a bounding function $V \colon E \to \mathbb{R}$ of \eqref{e:equation} (Definition \ref{d:bf}) which is locally Lipschitzian with $K \subset E$ and let there is $\delta \in (0,\infty)$ such that
\begin{equation}\label{e:S(t)K}
S(\tau)\overline K\subseteq \overline K \quad \text{for  } \tau\in [0,\delta].
\end{equation}
If $x \colon [a,b] \to E$ is a mild solution of \eqref{e:equation} with $x(t) \in \overline K$ for  $t \in [a,b]$, then $x(t) \in K$ for $t \in (a,b]$.
\end{theorem}

 \begin{proof} \ Let $x:[a,b] \to E$ be a solution of \eqref{e:equation} satisfying $x(t) \in \overline K $ for all $ t \in [a,b] $. Assume, by contradiction, the existence of $ \hat t \in (a,b] $ such that $\hat x:=x(\hat t) \in \partial K$.\\
By condition \emph{(V2)} there is a sequence $\{k_n\}\subset (-\infty, 0)$  with $\hat t +k_n>a$ for all $n$ such that $k_n \to 0^-$ as $n \to \infty$ and
\begin{equation}\label{e:V2s}
\lim_{n \to \infty}\frac{V(\hat x+k_nf(\hat t,\hat x))}{k_n}<0.
\end{equation}
According to the definition of mild solution (see Definition \ref{d:mild}) and the properties of the semigroup we have that
$$
\begin{array}{rl}
\hat x=x(\hat t)=&S(\hat t-a )x(a)+\int_a^{\hat t}S(\hat t-s)f(s, x(s))\, ds\\
\\
=&S(-k_n)\left[S(\hat t-a+k_n)x(a)+\int_a^{\hat t +k_n}S(\hat t-s+k_n)f(s, x(s)) \, ds\right]\\
&+ \int_{\hat t +k_n}^{\hat t}S(\hat t-s)f(s, x(s)) \, ds\\
\\
=&S(-k_n)x(\hat t+k_n) + \int_{\hat t +k_n}^{\hat t}S(\hat t-s)f(s, x(s))\, ds.
\end{array}
$$
Hence
\begin{equation}\label{e:sxa}
\hat x-S(-k_n)x(\hat t+k_n)=\int_{\hat t +k_n}^{\hat t}S(\hat t-s)f(s, x(s))\, ds, \enspace n \in \mathbb{N}.
\end{equation}
Since
\begin{equation*}
\int_{\hat t +k_n}^{\hat t}S(\hat t-s)f(s, x(s))\, ds=\int_{\hat t +k_n}^{\hat t}S(\hat t-s)f(\hat t, \hat x)\, ds+\int_{\hat t +k_n}^{\hat t}S(\hat t-s)\left[f(s, x(s))-f(\hat t, \hat x) \right]\, ds,
\end{equation*}
by means of the change of variables $\tau=\hat t -s$ introduced in the first integral on the right hand side, we have
\begin{equation}\label{e:sxb}
\begin{array}{rl}
\displaystyle\int_{\hat t +k_n}^{\hat t}S(\hat t-s)f(s, x(s))\, ds=& \int_0^{-k_n} S(\tau)f(\hat t, \hat x)\, d\tau \\
\displaystyle+&\int_{\hat t +k_n}^{\hat t} S(\hat t-s)\left[f(s, x(s))-f(\hat t, \hat x) \right]\, ds.
\end{array}
\end{equation}
Since both $f$ and $x$ are continuous functions in their respective domains, for every $\varepsilon >0$ there is $\sigma=\sigma(\varepsilon)>0$ such that
\begin{equation}\label{e:f}
\Vert f(t, x(t))-f(\hat t, \hat x)\Vert \le \varepsilon, \quad \mbox{for }\vert t-\hat t\vert \le \sigma.
\end{equation}
Hence there is $\overline n=\overline n(\sigma)$ such that \eqref{e:f} is satisfied for $t\in [\hat t+k_n, \hat t]$ and $n\ge \overline n$. Consequently, since $ \Vert S(t) \Vert \le 1 $ for $ t\ge0 $ by \emph{(A)}, for $n \ge \overline n$ we have
\begin{equation*}
\begin{array}{rl}
\frac{1}{-k_n}\left \Vert\int_{\hat t +k_n}^{\hat t}S(\hat t-s)\left[f(s, x(s))-f(\hat t, \hat x) \right]\, ds\right \Vert \le & \frac{1}{-k_n} \int_{\hat t +k_n}^{\hat t}\left \Vert S(\hat t-s)\right \Vert \left \Vert f(s, x(s))-f(\hat t, \hat x) \right \Vert \, ds\\
\\
\le &\frac{1}{-k_n}\varepsilon(-k_n)=\varepsilon.
\end{array}
\end{equation*}
It shows that
\begin{equation*}
\frac{1}{-k_n}\int_{\hat t +k_n}^{\hat t}S(\hat t-s)\left[f(s, x(s))-f(\hat t, \hat x) \right]\, ds \to 0, \quad \text{as } n \to \infty.
\end{equation*}
Notice (see e.g. \eqref{e:media}) that
\begin{equation*}
\frac{1}{-k_n}\int_0^{-k_n}S(\tau)f(\hat t, \hat x)\, d\tau \to f(\hat t, \hat x), \quad \text{as } n \to \infty.
\end{equation*}
Therefore, from conditions \eqref{e:sxa} and \eqref{e:sxb}, we obtain
\begin{equation*}
\frac{S(-k_n)x(\hat t +k_n)-\hat x}{k_n} \to f(\hat t, \hat x).
\end{equation*}
We have then showed the existence of $\{\sigma_n\}\subset E$, depending on $\{k_n\}$, such that $\sigma_n \to 0$ as $n \to \infty$ and satisfying

\begin{equation}\label{e:sxc}
S(-k_n)x(\hat t +k_n)=\hat x+k_nf(\hat t, \hat x)+k_n\sigma_n.
\end{equation}

\sm
Let $U \subset E$ be open with $\hat x \in U$ and $L>0$ be such that $V|_U$ is $L$-Lipschitzian. Take $n$ large enough in such a way that both  $\hat x+k_nf(\hat t, \hat x)$ and $ \hat x+k_nf(\hat t, \hat x)+k_n\sigma_n$ belong to $U$ and  $S(-k_n)x(\hat t +k_n) \in \overline K$ by \eqref{e:S(t)K}.  According to \emph{(V1)} and \eqref{e:sxc} we obtain that
$$
0\le \frac{V(S(-k_n)x(\hat t +k_n))}{k_n}=\frac{V(\hat x+k_nf(\hat t, \hat x)+k_n\sigma_n)}{k_n}= \frac{V(\hat x+k_nf(\hat t, \hat x))}{k_n}+\Delta_n,
$$
with
$$
\Delta_n:= \frac{V(\hat x+k_nf(\hat t, \hat x)+k_n\sigma_n)-V(\hat x+k_nf(\hat t, \hat x))}{k_n}.
$$
By the L-Lipschitzianity in $U$ of $ V, $ we obtain that
\begin{equation*}
\vert \Delta_n \vert \le L\Vert \sigma_n \Vert \to 0, \enspace \text{as }n \to \infty.
\end{equation*}
 Consequently,
\begin{eqnarray}\label{eq:contr1}
\lim_{n \to \infty} \frac{V(\hat x +k_nf(\hat t, \hat x))}{k_n}&=&\lim_{n \to \infty} \biggl[ \frac{V(\hat x+k_nf(\hat t, \hat x))}{k_n}+\Delta_n \biggr]\nonumber\\
&=& \lim_{n \to \infty} \frac{V(S(-k_n)x(\hat t +k_n))}{k_n} \ge 0
\end{eqnarray}
in contradiction with \eqref{e:V2s}. Hence $\hat x \in K$ and the proof is complete.
\end{proof}

\smallskip

 The case when the $0$-sublevel set of the bounding function is the ball centered in $0$ and with radius $ r,$ i.e.
\begin{equation}\label{e:Vr}
  V_r(x)=\frac 12 \left( \Vert x \Vert^2 -r^2\right), \enspace x\in E, \, r>0,
\end{equation}
frequently occurs in several applications (see e.g. Section \ref{s:application}).

\begin{Ex}\label{ex:Vnorma} \ \emph{(i)} Let $E$ be a Hilbert space with scalar product $\langle \cdot, \, \cdot \rangle$. In this case $V_r \in C^1(E)$ and $\dot V_r (x) \colon E \to E $ is such that $\dot V_r(x)(y)= \langle x, \, y \rangle, \, \,   x, y \in  E, \, r>0. $ Moreover  $\Vert \dot V_r(x) \Vert =\Vert x \Vert, \, x \in E, $ and then $ V_r $ is also locally Lipschitzian. Therefore, when the function $ f $ satisfies
\begin{equation}\label{e:V2H}
\langle x, \, f(t,x) \rangle <0, \enspace \text{for } t\in (a,b] \text{ and } \, \Vert x\Vert =r,
\end{equation}
then, by \eqref{e:CH}, $V_r$ is a locally Lipschitzian bounding function for \eqref{e:equation} for all $ r>0. $

\emph{(ii) } Assume, now, that $ E^* $ is a uniformly convex Banach space. Hence the function $ V_r $ is Frech\'{e}t differentiable on $ E $ with
\begin{equation*}
\langle \dot V_r(x), y \rangle=\langle J(x), y\rangle, \enspace \text{ for } x, y \in E,
\end{equation*}
where $ J \colon E \to E^* $ is the single-valued duality map given by
$$
J(x)=x^* \in E^* \; : \; \|x^*\|=\|x\| \; \mbox{and} \; \bigl\langle x^*, x \bigr\rangle =\|x\|^2
$$
and $ J $  is continuous (see e.g. \cite{D}). Hence, if we further assume the existence of $ r>0 $ such that
\begin{equation}\label{e:J}
\langle J(x), \, f(t,x) \rangle <0, \enspace \text{for } t\in (a,b] \text{ and } \, \Vert x\Vert =r,
\end{equation}
then $V_r$  is a locally Lipschitzian bounding function for equation \eqref{e:equation}. \\
In particular, let  $ E=L^p(\Omega) $ where $\Omega \subset \mathbb{R}^n$ is a bounded measurable subset of $ \mathbb{R}^n, n\ge 1 $  and $ 1<p<\infty. $ Then $ E $ is reflexive,  $ E^*=L^{p\prime}, $ with $ 1/p+1/p^{\prime}=1, $  is uniformly convex and (see  \cite[Example 1.4.4]{vr1} and \cite[Chapter I, Example 2.7]{HP})
\begin{equation*}
\langle J(x), y \rangle = \displaystyle\frac{1}{\|x\|^{p-2}} \displaystyle\int_\Omega |x(\xi)|^{p-2} \; x(\xi) \; y(\xi) \, d\xi, \enspace x, y \in L^p(\Omega).
\end{equation*}
Again by \eqref{e:J}, if the following condition
\begin{equation}\label{e:Jp}
\langle J(x), f(t,x) \rangle = \displaystyle\frac{1}{\|x\|^{p-2}} \displaystyle\int_\Omega |x(\xi)|^{p-2} \; x(\xi) \; f(t, x)(\xi) \, d\xi<0,  \enspace \text{for } t\in (a,b] \text{ and } \, \Vert x\Vert =r,
\end{equation}
is satisfied, then  $V_r$  is a locally Lipschitzian bounding function for \eqref{e:equation}.
\end{Ex}

\begin{Ex}\label{ex:Vd} \ In a uniformly convex Banach space $ E, $ consider the usual distance function $d(\cdot, r\overline B) \colon E \to \mathbb{R}$ from the closed ball centered in $0$ with radius $r>0$. From its definition,
\begin{equation*}
d(x, r\overline B) :=\left\{\begin{array}{ll}
0 &  \Vert x \Vert \le r\\
\displaystyle{\inf_{y\in r\overline B} } \, \Vert x-y\Vert=\Vert x \Vert  -r &  \Vert x \Vert >r
\end{array} \right.
\end{equation*}
 it is easy to show that it is a Lipschitzian function. If we further assume that
\begin{equation}\label{e:dist}
\liminf_{h \to 0^-}\frac{d(x+hf(t,x), r\overline B)}{h}<0, \qquad \text{for } t\in(a, b], \, \Vert x \Vert =r,
\end{equation}
the function $d(\cdot, \overline B)$ is a bounding function for equation \eqref{e:equation}.

\end{Ex}

\section*{Acknowledgments}
The authors are members of the {\em Gruppo Nazionale per l'Analisi Matematica, la Probabilit\`{a} e le loro Applicazioni} (GNAMPA) of the {\em Istituto Nazionale di Alta Matematica} (INdAM) and acknowledge financial support from this institution. The first author has been supported by the project Fondi Ricerca di Base 2016 \emph{Metodi topologici per equazioni differenziali in spazi astratti}, Department of Mathematics and Computer Science, University of Perugia.

\end{document}